\documentclass[12pt]{article}     \usepackage{amsmath}
\usepackage{amsthm}               \usepackage{amsopn}
\usepackage{graphicx}
\usepackage{latexsym}             \usepackage{amsfonts}
\usepackage{amssymb}   		      \usepackage{relsize}           
\usepackage{subcaption}
\usepackage{mathrsfs}
\swapnumbers
\theoremstyle{plain}
\newtheorem{theorem}{Theorem}[section]     
\newtheorem{proposition}[theorem]{Proposition}
\newtheorem{lemma}[theorem]{Lemma}

\theoremstyle{definition}
\newtheorem{definition}[theorem]{Definition}
\newtheorem{remark}[theorem]{Remark}

\DeclareRobustCommand{\rchi}{{\mathpalette\irchi\relax}}
\newcommand{\irchi}[2]{\raisebox{\depth}{$#1\chi$}}
\newcommand{\bigrchi}{\mathlarger{\mathlarger{\rchi}}}
\newcommand{\bcdot}{\boldsymbol{\cdot}}
\newcommand{\R}{\rm I \! R}
\newcommand{\newsec}[2]{ \section{#1} \label{sec-#2}  
                         \setcounter{equation}{0}     
                         \setcounter{theorem}{0} } 
\begin{document} 
\numberwithin{equation}{section}


\title{ Properties of Mean Value Sets: Angle Conditions, Blowup Solutions, and Nonconvexity}
\author{Niles Armstrong\\
\normalsize}

\maketitle

\begin{abstract}
We study the mean values sets of the second order divergence form elliptic operator with principal coefficients defined as $$a^{ij}_k(x):= \begin{cases}
\alpha_k \delta^{ij}(x)  &x_n>0 \\
\beta_k \delta^{ij}(x) &x_n<0. 
\end{cases}$$
In particular, we will show that the mean value sets associated to such an operator need not be convex as $\alpha_k$ and $\beta_k$ converge to 1. This example then leads to an example of nonconvex mean value sets for smooth $a^{ij}(x)$. 
\end{abstract}


\newsec{Introduction}{Intro}
	The mean value theorem has been one of the most important tools in the study of harmonic functions.  Due to this fact an analog for so called L-harmonic functions, where L is a generalized second order elliptic divergence form operator, is immediately of great interest.  One of the first such analogs can be found in \cite{LSW}, where a mean value theorem for such operators with no lower order terms was proved.  Using Equation 8.3 in their paper you can derive the following 
\begin{theorem}[Littman-Stampacchia-Weinerger]
Let $u$ be an $L$-subharmonic function on $\Omega$ with $u=0$ on $\partial \Omega$ and define
$$I(y;S):=\frac{S}{2} \int_{\frac{1}{S} \leq G \leq \frac{3}{S}} u(x) a^{ij}(x) D_{x_i}G(x,y) D_{x_j}G(x,y)dx$$
where $G(x,y)$ is the Green's function for $L$ on $\Omega$.  Then we have
$$u(y) \leq I(y;R) \leq I(y;S)  \text{ for all } 0<R<S.$$
\end{theorem} 
While this formula does provide us with a integral definition of $u(y)$ like the mean value theorem for harmonic functions does, it also has a number of issues.  First, it is not a simple average due to the presence of weights and indeed these weights involve derivatives of the Green's functions, which may be nontrivial to estimate.  Second, it is an average over level sets of the Green's function instead of an average over some nice set containing $y$.

The following theorem, which was stated by Caffarelli in \cite{C} and later proved in \cite{BH1} (still assuming no lower order terms), provides a simpler formula that is much closer to the original mean value theorem.
\begin{theorem} \label{MVT}
Let $L$ be any divergence form elliptic operator with ellipticity constants $\lambda$, $\Lambda$.  For any $x_0 \in \R^n$, there exists an increasing family $D_R(x_0)$ which satisfies the following:
\begin{enumerate}
\item $B_{cR}(x_0) \subset D_R(x_0) \subset B_{CR}(x_0)$, with $c$, $C$ depending only on $n$, $\lambda$, and $\Lambda$.
\item For any L-subharmonic function $v$ and $0<R<S$, we have
\begin{equation}
v(x_0) \leq \frac{1}{|D_R(x_0)|} \int_{D_R(x_0)} v \leq \frac{1}{|D_S(x_0)|} \int_{D_S(x_0)} v.
\end{equation}
\end{enumerate}

Finally, the sets $D_R(x_0)$ are noncontact sets of the following obstacle problem:

\begin{equation} \label{obprob1}
\begin{cases}
Lu=-\bigrchi_{\{u<G(\cdot,x_0)\}}R^{-n} &\text{in} \quad B_M(x_0) \\

u \leq G(\cdot, x_0) &\text{in} \quad B_M(x_0) \\

u=G(\cdot,x_0) &\text{on} \quad \partial B_M(x_0)
\end{cases}
\end{equation}
where $B_M(x_0) \subset \R^n$ and $M>0$ is sufficiently large and $G(x,y)$ is the Green's function for $L$ on $\R^n$.
\end{theorem}

\begin{definition}
We defined the mean value sets associated to an operator $L$ to be the sets $D_R(x_0)$ in Theorem \ref{obprob1}.  When it is clear we may simply refer to these sets as the mean value sets.
\end{definition}

\begin{remark} \label{bndry}
Throughout the paper $u=g$ on $\partial \Omega$ we take to be interpreted as $u-g \in W^{1,2}_0(\Omega)$, but we frequently have a technical issue because $g$ will often be a Green's function which does not belong to $W^{1,2}$.  On the other hand, because our Green's functions belong to $W^{1,2}$ away from the singularity which in turn lies in the interior, we can follow \cite{BH1} and truncate $G$ near the singularity.  (See \cite{BH1} for details about this technicality.)
\end{remark}

Note the key difference between the MVT for the Laplacian and Theorem \ref{MVT} is the appearance of the sets $D_R(x_0)$ as apposed to $B_R(x_0)$.  Initially all that was known about these mean value sets $D_R(x_0)$ can be seen in the previous theorem along with Corollary 3.10 in \cite{BH1} which states that $\partial D_R(x_0)$ is porous, so in particular has Hausdorff dimension strictly less than $n$.  More recently the following three properties of these sets were shown in \cite{AB}:
 
\begin{enumerate}
\item If $y_0 \neq x_0$ then there exists an $R>0$ such that $y_0 \in \partial D_R(x_0)$.
\item Assume $y_0 \in \partial D_R(x_0)$ and take $c$ to be as in Theorem \ref{MVT}. For any $h\in(0,1/2),$ the density 
$$\frac{|D_R(x_0) \cap B_{chR}(y_0)|}{|B_{chR}(y_0)|}$$
is bounded from below by a positive constant depending only on $n$, $\lambda$, and $\Lambda$.
\item For any $x_0 \in \R^n$ and for any $R>0$ such that $B_{CR} \subset \R^n$, the set $D_R(x_0)$ has exactly one component.
\end{enumerate}
At the same time Theorem \ref{MVT} was extended to the Laplace-Beltrami operator on Riemannian manifolds in \cite{BBL}.

	Additionally the topology and geometry of these mean value sets is largely still unknown.  As in \cite{AB} one can ask if such sets will be convex, star-shaped, or homeomorphic to a ball.  However, what might be more reasonable is to ask for what types of operators, L, would such sets have a given property.  In this way we aim to study a very particular type of operators.  Namely operators whose principal coefficients are defined as
$$a^{ij}(x):= \begin{cases}
\alpha \delta^{ij}(x)  &x_n>0 \\
\beta \delta^{ij}(x) &x_n<0 
\end{cases} $$
for some $\alpha$, $\beta>0$.  We have two main reasons for studying such operators.  First these operators appear in the study of composite materials where one has two different constituent materials.  (See \cite{LV}).  Secondly these $a^{ij}(x)$ have, what we view to be, the simplest type of discontinuity for such coefficients.  Indeed, we thought these operators to be the simplest in which computing the mean value sets would be nontrivial, but still possible.  In fact if $\alpha$ is very close to $\beta$ such an operator would be very close to the Laplacian.  Hence, it is reasonable to assume the mean value sets would be very close to Euclidean balls.  This intuition does in fact turn out to be true. In particular, in Theorem \ref{conv} we show that 
$$|D_R(x_0) \, \Delta \, B_R(x_0)| \rightarrow 0 \text{ as } \alpha, \, \beta \rightarrow 1. $$
On the other hand, this fact makes our main result somewhat surprising, where we show that as $\alpha$ tends towards $\beta$  almost every $R$ such that $\partial D_R(x_0)$ crosses the interface, $D_R(x_0)$ will be nonconvex.  One may think that this behavior is due to the discontinuous nature of the principal coefficients, but this idea is incorrect.  In fact we will prove the existence of nonconvex mean value sets corresponding to operators whose principal coefficients are smooth on all of $\R^n$.  Moreover, this suggests that convexity may only be possible for operators which are the Laplacian after a linear change of variables.

\begin{theorem} \label{nonconv}
Let $D_{R;k}(x_0,y_0) \subset \R^2$ be as in Theorem \ref{MVT} where the principal coefficients of $L_{\alpha, \beta}$ are defined as
$$a^{ij}_k(x,y):= \begin{cases}
\alpha_k \delta^{ij}(x,y)  &y>0 \\
\beta_k \delta^{ij}(x,y) &y<0 
\end{cases}$$
with $\alpha_k$ and $\beta_k$ converging to 1 as $k\rightarrow\infty$. Then for all most every choice of $y_0$, such that $\partial D_{R;k}(x_0,y_0)$ eventually crosses the interface, there exists a constant $K>0$ such that $D_{R;k}(x_0,y_0)$ is nonconvex for all $k > K$.
\end{theorem}	

The remainder of the paper will be organized in the following way.  In section 2 we will state a few preliminary results that will be needed throughout the paper and define much of the notation we will be using.  Section 3 is devoted to proving a measure stability result, \emph{(4)} in Theorem \ref{conv}, which states that our mean value sets are close to Euclidean balls in measure when the operator L is close to the Laplacian.  For section 4 we prove a Weiss' type monotonicity formula, Theorem \ref{mono}, that leads to homogeneous of degree two blowup solutions at free boundary points on the interface.  We note here that other such quasi-monotonicity formulas have been derived for rather general $a^{ij}(x)'s$, which lead to many of the same results to that of a full monotonicity formula.  In \cite{FGS} Theorem 3.7 provides such a formula where $a^{ij}(x)$ are Lipschitz continuous and in \cite{G} Theorem 1.1 provides such a formula where $a^{ij}(x)$ are in a fractional Sobolev space.  While these formulas are shown with few structural assumptions on the $a^{ij}(x)'s$ the mild regularity that is required excludes the $a^{ij}(x)'s$ we will look at here.  In section 5 we prove an angle condition for blowup solutions at free boundary points on the interface, Lemma \ref{anglecon}.  This condition then leads to the proof of Theorem \ref{nonconv}.  Finally, in section 6 we show an analog of Theorem \ref{nonconv} where the principal coefficients of the operator are no longer discontinuous.  This is done by convolving the discontinuous $a^{ij}(x)'s$ with a mollifier and showing that the convergence results found in section 3 still hold in this case. 

\newsec{Preliminaries and Terminology}{PreTerm}
We will use the following basic notation throughout the paper: \\

\begin{tabular}{l l}
$\bigrchi_D$ & the characteristic function of the set D \\

$\overline{D}$ & the closure of the set D \\

$int(D)$	& the interior of the set D \\

$\partial D$ & the boundary of the set D \\

$D^c$ & the complement of the set D \\

$x$ & $(x_1, x_2, \ldots, x_n)$ \\

$x'$ & $(x_1, x_2, \ldots, x_{n-1})$ \\

$B_r(x)$ & the open ball with radius r centered at the point x \\

$B_r$ & $B_r(0)$ \\

$B^+_r(x)$, $B^-_r(x)$ & $\{y \in B_r(x) \, | \, y_n>0\}$, $\{y \in B_r(x) \, | \, y_n<0\}$ \\

$\partial^+B_r(x)$, $\partial^-B_r(x)$ & $\{y \in \partial B_r(x) \, | \,  y_n>0\}$, $\{y \in \partial B_r(x) \, | \, y_n<0\}$ \\

$\R^n_+$, $\R^n_-$ & $\{x \in \R^n \, | \, x_n \geq 0\}$, $\{x \in \R^n \, | \, x_n \leq 0\}$ \\

$\Omega(w)$ & $\{w>0\}$ \\

$\Lambda(w)$ & $\{w=0\}$ \\

$FB(w)$ & $\partial \Omega(w) \cap \partial \Lambda(w)$ \\

$\Delta$ & Laplacian or symmetric difference operator. \\ 
& In the second case $A \, \Delta \, B:=(A \setminus B) \cup (B \setminus A)$. \\

$\Gamma(x,y)$ & the fundamental solution for the Laplacian on $\R^n$. \\ 
& (See \cite{GT} and \cite{HL}). \\
$L_{\alpha, \beta}, L_k$ & operators found in Equations (\ref{oper}) and (\ref{obprob2}) \\

$d \mathcal{H}^{n-1}$ & $(n-1)$-dimensional Hausdorff measure 

\end{tabular}
\\ \\
In the entirety of the paper we will work with a divergence form elliptic operator 
\begin{equation} \label{oper}
L_{\alpha,\beta}:=D_ja^{ij}(x)D_i
\end{equation}
with the $a^{ij}(x)$ having the following structure
\begin{equation} \label{aij}
a^{ij}(x):= \begin{cases}
\alpha \delta^{ij}(x)  &x_n>0 \\
\beta \delta^{ij}(x) &x_n<0 
\end{cases} \quad \text{for} \qquad 0<\lambda \leq \alpha, \beta \leq \Lambda
\end{equation}
where $\delta^{ij}$ is the Kronecker delta function.  From here out we will refer to the set $\{x_n=0 \}$ as the interface.  Solutions here will be meant in the weak sense, as in, for a function $u \in W^{1,2}(\Omega)$ and $f \in L^2(\Omega)$ we take ``$Lu=f$'' to mean
\begin{equation}
-\int_{\Omega} a^{ij}(x)D_i u D_j \phi = \int_{\Omega} f \phi \qquad \text{for any } \phi \in W^{1,2}_0(\Omega).
\end{equation}
For such operators the Green's function on $\R^n$ have been computed explicitly in \cite{ABH}; we will restate them here for the convenience of the reader.  Let $\widetilde{\Gamma}(x,y):=\Gamma((x',-x_n),y)$.  Then the Green's function for $L_{\alpha,\beta}$ on $\R^n$ is as follows

$$
G(x,y):=
\begin{cases}
G_h(x,y) &\text{when } y_n>0 \\
G_l(x,y) &\text{when } y_n<0 \\
\end{cases}
$$
where
$$
G_h(x,y):=\frac{1}{\alpha} \Gamma(x,y)+
\begin{cases}
\frac{\alpha-\beta}{\alpha(\alpha+\beta)}\widetilde{\Gamma}(x,y) &\text{when } x_n \leq 0 \\
\frac{\alpha-\beta}{\alpha(\alpha+\beta)}\Gamma(x,y) &\text{when } x_n \geq 0 \\
\end{cases}
$$
and
$$
G_l(x,y):=\frac{1}{\beta}\Gamma(x,y)+
\begin{cases}
\frac{\beta-\alpha}{\beta(\alpha+\beta)}\Gamma(x,y) &\text{when } x_n \leq 0 \\
\frac{\beta-\alpha}{\beta(\alpha+\beta)}\widetilde{\Gamma}(x,y) &\text{when } x_n \geq 0. \\
\end{cases}
$$
When $n=2$ we will abuse the vocabulary slightly and still refer to such a $G(x,y)$ as the Green's function, where in this case the limit at $\infty$ is $-\infty$ instead of $0.$  We will also need a well known transmission condition for $L_{\alpha,\beta}$-harmonic functions in this setting, which we state here.

\begin{lemma} [Transmission Conditions] \label{TCon}
With $a^{ij}(x)$ and $L_{\alpha, \beta}$ definded as above, if $L_{\alpha, \beta} w=0$ in $ \R^n $, then $w \in C^{\omega}(\R^n_+) \cap C^{\omega}(\R^n_-)$ and for any $x' \in \R^{n-1}$ we have

\begin{equation} \label{InterCon}
\alpha \lim_{s \downarrow 0} \frac{\partial}{\partial x_n} w(x',s) = \beta \lim_{t \uparrow 0} \frac{\partial}{\partial x_n} w(x',t).
\end{equation}
\end{lemma}

The regularity of the solution $w$ in this lemma was shown by Li and Vogelius in the appendix of \cite{LV},  whereas the condition in (\ref{InterCon}) can be derived by doing an integration by parts with standard arguments from calculus of variations.  Note that Lemma \ref{TCon} gives us solutions that are real analytic up to and including the interface, but not across the interface.

	
\newsec{Measure Stability}{Meas Stab}
This section will show the convergence of solutions to obstacle problems of the form (\ref{obprob1}), with the operator $L_{\alpha, \beta}$, as $\alpha$ and $\beta$ converge to 1.  This fact is then used to show a type of measure convergence for the mean value sets associated to the operator, $L_{\alpha, \beta}$, to those of the Laplacian, i.e. Euclidean balls.  In this sense the mean value sets do begin to act like Euclidean balls.

Now we fix some $x_0 \in \R^n$ and consider $w_k \in W^{1,2}(B_{16M}(x_0))$ to be the solutions to
\begin{equation} \label{obprob2}
\begin{cases}
L_ku:=D_i(a^{ij}_k(x) D_j u)=-R^{-n} \bigrchi_{\{u<G_k(\bcdot,x_0)\}} &\text{in } B_{16M}(x_0) \\
u \leq G_k(\bcdot,x_0)	&\text{in } B_{16M}(x_0) \\
u=G_k(\bcdot,x_0) &\text{on } \partial B_{16M}(x_0)
\end{cases}
\end{equation}
where $L_k:=L_{\alpha_k,\beta_k}$ with ellipticity constants $0<\lambda<\Lambda$ for all $k.$  We will choose $M>0$ large enough so that the solutions are independent of $M$, see \cite{BH1} for details.  In fact, we will choose $M$ large enough to ensure that $\{w_k<G_k(\bcdot,x_0)\} \subset B_M(x_0)$ for all $k$.

We now wish to show two compactness results for the solutions $w_k$.  These results are very similar in proof and statement to that found in \cite{BH2} Lemma 3.1 and Lemma 3.2.  A key difference being we can not assume $0 \in FB(w_k)$ for all $k$.  To accommodate this change we first prove a uniform interior $L^{\infty}$ bound on the solutions.  Another discrepancy here is that we will not be able to utilize the height function $G_k(\bcdot,x_0)-w_k$ due to the singularity from the Green's function.

\begin{lemma}[Uniform Bound] \label{UniBdd}
There exists a constant $C$ such that $w_k \leq C$ in $B_{4M}(x_0)$ for all $k$.
\end{lemma}
\begin{proof}
By lemma 4.4 from \cite{BH1} we have 
\begin{equation*}
w_k \leq b_k+K \Big(b_k+\frac{R^{-n}M^2}{4n} \Big) \text{ in } B_{4M}(x_0) \text{ where } b_k:=\max\limits_{\partial B_{16M}(x_0)} G_k(\bcdot,x_0).
\end{equation*}
Using $G_k(\bcdot,x_0) \leq \frac{1}{\lambda}\Gamma(\bcdot,x_0)$ gives
\begin{equation*}
b_k \leq b:=\max\limits_{\partial B_{16M}(x_0)} \frac{1}{\lambda}\Gamma(\bcdot,x_0). 
\end{equation*}
Hence, $w_k \leq b+K \Big(b+\frac{R^{-n}M^2}{4n} \Big)$ for all $k$.
\end{proof}

\begin{lemma}[Compactness] \label{comp}
There exists a function $w \in W^{1,2}(B_{2M}(x_0))$ and a subsequence of \{$w_k$\} such that along this subsquence we have uniform convergence of $w_k$ to $w$, and weak convergence in $W^{1,2}$.
\end{lemma}
\begin{proof}
Lemma \ref{UniBdd} along with De Giorgi-Nash-Moser theorem implies there exists an $\alpha \in (0,1)$ and a constant $C$ such that $\{w_k\} \subset C^{\alpha}( \overline{B_{2M}(x_0)} )$ with $||w_k||_{C^{\alpha}( \overline{B_{2M}(x_0)} )} 
\leq C$.  Then by Arzela-Ascoli there exists a subsequence of $\{w_k\}$ that converges uniformly, to $w \in C^0(\overline{B_{2M}(x_0)})$.  Also, standard elliptic regularity plus the uniform $L^2$ bound on $G_k(\bcdot,x_0)$ implies a uniform $W^{1,2}$ bound on $w_k$.  Then by standard functional analysis there exists a subsequence such that $w_k \rightharpoonup w$ in $W^{1,2}$.
\end{proof}

\begin{theorem} \label{conv}
If we now assume that
$$\alpha_k \text{ and } \beta_k \rightarrow 1 \qquad \text{in } B_M(x_0)$$
we then have
\begin{enumerate}
\item[(1)] $G_k(\bcdot,x_0) \rightarrow  \Gamma(\bcdot,x_0)$ uniformly in compact subsets of $\R^n \setminus \{x_0\}$
\item[(2)] $ \bigrchi_{\{w_k<G_k(\bcdot,x_0)\}} \rightarrow  \bigrchi_{\{w<\Gamma(\bcdot,x_0)\}}$ almost everywhere in $B_M(x_0)$
\item[(3)] $w$ satisfies
$$
\begin{cases}
\Delta w=-R^{-n} \bigrchi_{\{w<\Gamma(\bcdot,x_0)\}} &\text{in } B_{M}(x_0) \\
w \leq \Gamma(\bcdot,x_0) &\text{in } B_{M}(x_0) \\
w = \Gamma(\bcdot,x_0) &\text{on } \partial B_{M}(x_0)
\end{cases}
$$
\item[(4)] $|\{ w_k<G_k(\bcdot,x_0) \} \, \Delta \, \{ w<\Gamma(\bcdot,x_0 \}| \rightarrow 0$
\end{enumerate}
where $w$ is the limiting function from the previous lemma.
\end{theorem}
\begin{proof}[Proof of (1)]
Let
$$\Psi_k(\bcdot,x_0):=
\begin{cases}
\frac{1}{\alpha_k}\Gamma(\bcdot,x_0)	&\text{when } x_n>0 \\
\frac{1}{\beta_k}\Gamma(\bcdot,x_0)	&\text{when } x_n<0
\end{cases}
 $$
Then we have 
\begin{equation*}
|G_k(\bcdot,x_0)-\Gamma(\bcdot,x_0)| \leq |G_k(\bcdot,x_0)-\Psi_k(\bcdot,x_0)| +|\Psi_k(\bcdot,x_0)-\Gamma(\bcdot,x_0)|
\end{equation*}
which converges to $0$ uniformly on any compact subset of $\R^n \setminus \{x_0\}$.
\\ \\
\emph{Proof of (2).}
Note that from \emph{(1)} and Lemma \ref{comp} we have 
\begin{equation*}
(G_k(\bcdot,x_0)-w_k) \rightarrow (\Gamma(\bcdot,x_0)-w) \text{ in } B_M(x_0) \setminus \{x_0\}. 
\end{equation*}
Hence, in the interior of $\{ \Gamma(\bcdot,x_0)-w>0 \}$ we have 
\begin{equation*}
\bigrchi_{\{w_k<G_k(\bcdot,x_0)\}} \rightarrow \bigrchi_{\{w<\Gamma(\bcdot,x_0)\}} \text{ on } B_M(x_0) \setminus \{x_0\}. 
\end{equation*}
The same is true in the interior of $\{ \Gamma(\bcdot,x_0)-w=0 \}$ by using the nondegeneracy of our solutions, which can be found in \cite{BH1}.  Finally, we note that $\partial \{ \Gamma(\bcdot,x_0)-w=0 \}$ has Lebesgue measure zero from \cite{BH1}.
\\ \\
\emph{Proof of (3).}
For any $\phi \in W^{1,2}_0(B_M(x_0))$,
$$\int_{B_M(x_0)} a^{ij}_k D_i w_k D_j \phi = \int_{B_M(x_0)} (a^{ij}_k-\delta^{ij})(D_i w_k - D_i w) D_j \phi$$

$$+ \int_{B_M(x_0)} a^{ij}_k D_i w D_j \phi + \int_{B_M(x_0)} \delta^{ij} (D_i w_k - D_i w) D_j \phi=:I+I\!I+I\!I\!I. $$

Since $a^{ij}_k \rightarrow \delta^{ij}$ and $D_i w_k \rightharpoonup D_i w$, we have
$$I=\int_{B_M(x_0)} (a^{ij}_k-\delta^{ij})(D_i w_k - D_i w) D_j \phi \rightarrow 0$$

$$I\!I=\int_{B_M(x_0)} a^{ij}_k D_i w D_j \phi \rightarrow \int_{B_M(x_0)} \delta^{ij} D_i w D_j \phi $$
and
$$I\!I\!I=\int_{B_M(x_0)} \delta^{ij} (D_i w_k - D_i w) D_j \phi \rightarrow 0.$$
Hence,
$$\int_{B_M(x_0)} a^{ij}_k D_i w_k D_j \phi \rightarrow \int_{B_M(x_0)} \delta^{ij} D_i w D_j \phi. $$
Together with \emph{(2)}, we proved
$$\Delta w= \bigrchi_{\{w<\Gamma(\bcdot,x_0)\}} \quad \text{in } B_M(x_0). $$
From \emph{(1)} and Lemma \ref{comp} it is clear that $w \leq \Gamma(\bcdot,x_0)$ in $B_M(x_0)$.  Finally, recall we have $w_k=G_k(\bcdot,x_0)$ on $\partial B_M(x_0)$ then again from \emph{(1)} and Lemma \ref{comp} we have $w= \Gamma(\bcdot,x_0)$ on $\partial B_M(x_0)$.
\\ \\
\emph{Proof of (4).}
This immediately follows from \emph{(2)}.
\end{proof}

\begin{remark}
Due to the result in \cite{K} we know $B_R(x_0)=\{w<\Gamma(\cdot,x_0) \}$.  Hence, \emph{(4)} from the previous theorem can be rewritten to say
$$|\{w_k<G_k(\bcdot,x_0)\} \, \Delta \, B_R(x_0)| \rightarrow 0. $$
In this sense we have shown our mean value sets converge to Euclidean balls.
\end{remark}

\newsec{Monotonicity Formula}{Mono Form}
In this section we aim to adapt Weiss' monotonicity formula, from \cite{W}, to solutions of the following problem

\begin{equation} \label{pde}
\begin{cases}
L_{\alpha, \beta} u = \frac{1}{2}  \bigrchi_{\{u>0\}} & \text{in } B_1(0) \\
u \ge 0 & \text{in } B_1(0)\\
0 \in FB(u).
\end{cases}
\end{equation}

\begin{remark}
Note that Lemma \ref{TCon} implies a solution, $w$, to (\ref{pde}) will satisfy
$$\alpha \Delta w= \frac{1}{2} \, \bigrchi_{\{w>0\}} \qquad \text{in } \R^n_+$$ 
$$\beta \Delta w= \frac{1}{2} \, \bigrchi_{\{w>0\}} \qquad \text{in } \R^n_- $$
in a classical sense.
\end{remark}

We can now state and prove the desired monotonicity formula.  Note that the theorem requires $x\in \{x_n=0\}$.  While this may seem like a big restriction, realize that off of the interface Weiss' original monotonicity formula would hold.

\begin{theorem}
\label{mono}
Let $w$ be a solution to (\ref{pde}), $f(x):=\alpha \bigrchi_{\{x_n>0\}}+\beta \bigrchi_{\{x_n<0\}}$, and $x_0 \in \{ x_n=0 \}$. Then the function

$$\Phi_{x_0}(r):=r^{-n-2}\int_{B_r(x_0)} (f |\nabla w|^2+w)dx-2r^{-n-3}\int_{\partial B_r(x_0)} f w^2 d \mathcal{H}^{n-1}$$
defined in $(0,dist(x_0,\partial B_1(0))$, satisfies the monotonicity formula
$$\Phi_{x_0}(\sigma)-\Phi_{x_0}(\rho)=\int^{\sigma}_{\rho} r^{-n-2} \int_{\partial B_r(x_0)} 2 \Big( \nabla w \cdot \nu - 2 \frac{w}{r} \Big)^2f \, d\mathcal{H}^{n-1} dr \ge 0$$
for all $0< \rho < \sigma < 1$.
\end{theorem}
\begin{remark}
The key difference in our monotonicity formula from Weiss' original is the necessity of the function $f$ in the definition of $\Phi$.  This change is needed to reflect that our solution $w$ minimizes 
$$\int_D (f |\nabla u|^2+u)dx \quad \text{instead of} \quad \int_D (|\nabla u|^2+u)dx$$
among all functions $u \in W^{1,2}(D)$ with $u \geq 0$.
\end{remark}
\begin{proof}
We will omit the measures $dx$ and $d \mathcal{H}^{n-1}$ throughout the proof.  It is to be understood that the measure is $d \mathcal{H}^{n-1}$ when integrating over the boundary of a set and is otherwise $dx$, which is the Lebesgue n-dimensional measure.  

Let $w_r(x):=\frac{w(x_0+rx)}{r^2}$ and observe that
$$ \Phi(r)=\int_{B_1(0)} (f |\nabla w_r|^2+w_r)-2\int_{\partial B_1(0)} f w_r^2 .$$
Since we have $w \in C^{\omega}(\R^n_+) \cap C^{\omega}(\R^n_-)$ from Lemma \ref{TCon}, it is convenient to work with $\Phi$ in the form

$$\Phi(r)=\int_{B^+_1(0)} (\alpha |\nabla w_r|^2+ w_r)+\int_{B^-_1(0)} (\beta |\nabla w_r|^2+ w_r)$$
$$
-2\int_{\partial^+B_1(0)}\alpha w_r^2 -2\int_{\partial^-B_1(0)}\beta w_r^2.$$
Now compute

$$ \Phi'(r)=\int_{B^{+}_1(0)} 2 \alpha \nabla w_r \cdot \nabla \bigg( \frac{x \cdot \nabla w(x_0+xr)}{r^2}-\frac{2}{r}\frac{w(x_0+rx)}{r^2} \bigg)$$

$$+\int_{B^{-}_1(0)} 2 \beta \nabla w_r \cdot \nabla \bigg( \frac{x \cdot \nabla w(x_0+xr)}{r^2}-\frac{2}{r}\frac{w(x_0+rx)}{r^2} \bigg)$$

$$ + \int_{B^{+}_1(0)}  \bigg( \frac{x \cdot \nabla w(x_0+xr)}{r^2}-\frac{2}{r}\frac{w(x_0+rx)}{r^2} \bigg)$$

$$+ \int_{B^{-}_1(0)} \bigg( \frac{x \cdot \nabla w(x_0+xr)}{r^2}-\frac{2}{r}\frac{w(x_0+rx)}{r^2} \bigg)$$

$$ -2 \int_{\partial^+ B_1(0)} 2 \alpha w_r \bigg( \frac{x \cdot \nabla w(x_0+xr)}{r^2}-\frac{2}{r}\frac{w(x_0+rx)}{r^2} \bigg)$$

$$ -2 \int_{\partial^- B_1(0)} 2 \beta w_r \bigg( \frac{x \cdot \nabla w(x_0+xr)}{r^2}-\frac{2}{r}\frac{w(x_0+rx)}{r^2} \bigg).$$
An integration by parts on the first two terms yields:

$$ = -\frac{2}{r} \int_{B^+_1(0)} \alpha \Delta w_r (x \cdot \nabla w_r -2w_r)-\frac{2}{r} \int_{B^-_1(0)} \beta \Delta w_r (x \cdot \nabla w_r -2w_r)$$

$$+\frac{2 \alpha}{r}\int_{\partial^+ B_1(0)} x \cdot \nabla w_r (x \cdot \nabla w_r -2w_r)+\frac{2 \beta}{r}\int_{\partial^- B_1(0)} x \cdot \nabla w_r (x \cdot \nabla w_r -2w_r)$$

$$ +\frac{2\alpha}{r}\int_{B_1\cap\{ x_n=0 \}}(-e_n) \cdot \nabla w_r (x \cdot \nabla w_r -2w_r)$$

$$+\frac{2\beta}{r}\int_{B_1\cap\{ x_n=0 \}}(e_n) \cdot \nabla w_r (x \cdot \nabla w_r -2w_r) $$

$$ +\frac{1}{r} \int_{B^+_1(0)}( x \cdot \nabla w_r-2w_r)+\frac{1}{r} \int_{B^-_1(0)}( x \cdot \nabla w_r-2w_r)$$

$$-\frac{2\alpha}{r}\int_{\partial^+ B_1(0)} 2w_r(x \cdot \nabla w_r-2w_r)-\frac{2\beta}{r}\int_{\partial^- B_1(0)} 2w_r(x \cdot \nabla w_r-2w_r)$$

$$ =\frac{2\alpha}{r}\int_{\partial^+ B_1(0)} (x \cdot \nabla w_r-2w_r)^2+\frac{2\beta}{r}\int_{\partial^- B_1(0)} (x \cdot \nabla w_r-2w_r)^2$$

$$ +\frac{2\alpha}{r}\int_{B_1\cap\{ x_n=0 \}}(-e_n) \cdot \nabla w_r (x \cdot \nabla w_r -2w_r)$$

$$+\frac{2\beta}{r}\int_{B_1\cap\{ x_n=0 \}}(e_n) \cdot \nabla w_r (x \cdot \nabla w_r -2w_r).$$
The last two terms cancel due to the transmission condition Lemma \ref{TCon}.  Hence, we obtain

$$=\frac{2}{r}\int_{\partial B_1(0)} f(x \cdot \nabla w_r -2w_r)^2 d\mathcal{H}^{n-1}.$$
After rescaling we arrive at the desired result
$$ \Phi'_{x_0}(r)=2r^{-n-2}\int_{\partial B_r(x_0)} f \bigg(\nu \cdot \nabla w -2\frac{w}{r}\bigg)^2 d\mathcal{H}^{n-1} \ge 0$$
for $r \in (0, 1).$
\end{proof}

Similar to Weiss' original paper \cite{W} we now use our monotonicity formula to prove blowup solutions of (\ref{pde}) at free boundary points are homogeneous of degree two.
\begin{proposition} \label{deg2}
Let $w$ be a solution to (\ref{pde}) then we have
\begin{enumerate}
\item[(1)] For all $x_0 \in B_1(0) \cap \{ w=0 \} \cap \{x_n=0\}$ the function $\Phi_{x_0}(r)$ has a real right limit $\Phi_{x_0}(0^+)$
\item[(2)] Let $x_0 \in B_1(0) \cap \{w=0\} \cap \{x_n=0\}$ and $0<\rho_m \longrightarrow 0$ be a sequence such that the blow-up sequence $w_m(x):=\frac{w(x_0+\rho_mx)}{\rho^2_m}$ converges a.e. in $\R^n$ to a blow-up limit $w_0$.  Then $w_0$ is a homogeneous function of degree 2.
\item[(3)] $\Phi_{x_0}(r) \ge 0$ for every $x_0 \in B_1(0) \cap \{w=0\} \cap \{x_n=0\}$ and every $0 \le r <1$.  Equality holds if and only if $w=0$ in $B_r(x_0)$.
\item[(4)] The function $x_0 \longmapsto \Phi_{x_0}(0^+)$  restricted to the set $ \{x_n=0\} $ is upper semi-continuous. 
\end{enumerate}
\end{proposition}

\begin{proof}[Proof of (1)]
Since $w \in C^{\omega}(\R^n_+) \cap C^{\omega}(\R^n_-)$ we get that $\Phi_{x_0}(r)$ is bounded and non-decreasing from Theorem \ref{mono}, assuming that r is sufficiently small.
\\ \\
\emph{Proof of (2).} First we may assume $\rho_m$'s are sufficiently small, depending on $x_0$ so that Theorem \ref{mono} holds.  Then for all $0<R<S< \frac{1}{\rho_m}$ we have
$$ \int^S_R r^{-n-2} \int_{\partial B_r(0)} 2 \bigg(\nabla w_m \cdot \nu -2 \frac{w_m}{r}\bigg)^2 f \, d \mathcal{H}^{n-1}dr$$

$$= \int^{S\rho_m}_{R\rho_m} \delta^{-n-2} \int_{\partial B_{\delta}(0)} 2 \bigg(\nabla w \cdot \nu -2 \frac{w}{\delta}\bigg)^2 f \, d \mathcal{H}^{n-1}d\delta=\Phi_0(S\rho_m)-\Phi_0(R\rho_m)$$

where $\delta=r \rho_m$.  From \emph{(1)} we have $\Phi_{x_0}(S\rho_m)-\Phi_{x_0}(R\rho_m) \longrightarrow 0$.  This together with the boundedness of $\{w_m\} \subset C^{1,1}_{loc} (\R^{n}_+) \cap C^{1,1}_{loc} (\R^{n}_-) $ implies $w_0$ is homogeneous of degree 2.
\\ \\
\emph{Proof of (3).} First suppose $\Phi_{x_0}(0^+) < 0$ for some $x_0 \in \{x_n=0\} \cap \{ w=0 \}$.  Then $\exists$ a sequence of positive real values $\rho_m \longrightarrow 0$ and $w_m:=\frac{w(x_0+\rho_mx)}{\rho_m^2} \longrightarrow w_0$ in $C^{1,\alpha}_{loc} (\R^{n}_+) \cap C^{1,\alpha}_{loc} (\R^{n}_-)$ such that

$$0 > \lim_{m\rightarrow \infty} \Phi_{x_0}(\rho_m)$$

$$=\lim_{m\rightarrow\infty} \bigg[ \int_{B^+_1(x_0)} (\alpha |\nabla w_m|^2+w_m)+\int_{B^-_1(x_0)} (\beta |\nabla w_m|^2+ w_m)$$

$$-2\int_{\partial^+B_1(0)}\alpha w_m^2 d\mathcal{H}^{n-1} -2\int_{\partial^-B_1(0)}\beta w_m^2 d\mathcal{H}^{n-1} \bigg]$$

$$= \int_{B^+_1(x_0)} (\alpha |\nabla w_0|^2+\alpha w_0)+\int_{B^-_1(x_0)} (\beta |\nabla w_0|^2+\beta w_0)$$

$$-2\int_{\partial^+B_1(0)}\alpha w_0^2 d\mathcal{H}^{n-1} -2\int_{\partial^-B_1(0)}\beta w_0^2 d\mathcal{H}^{n-1}.$$

From \emph{(2)} we know that $w_0$ is homogeneous of degree 2.  An integration by parts yields

$$0>\int_{B^+_1}  w_0(1-\alpha \Delta w_0)+\int_{B^-_1}  w_0(1-\beta \Delta w_0)$$ 

$$+ \int_{\partial^+ B_1(0)} \alpha w_0(\nabla w_0 \cdot x -2w_0) d\mathcal{H}^{n-1}+ \int_{\partial^- B_1} \beta w_0(\nabla w_0 \cdot x -2w_0) d\mathcal{H}^{n-1}$$

$$=\frac{1}{2}\int_{B_1(0)} w_0 \ge 0,$$ 
a contradiction.  Hence, $0  \leq \Phi_{x_0}(0^+) \leq \Phi_{x_0}(r)$.  Finally if $w=0$ in $B_r(x_0)$ then obviously $\Phi_{x_0}(r)=0.$ On the other hand if $\Phi_{x_0}(r)=0$ then $\Phi'_{x_0}(r)=0$, which implies $w$ is homogeneous of degree 2.  Then following as before
$$0=\Phi_{x_0}(r)=\frac{1}{2}r^{-n-2}\int_{B_r} w \ge 0 .$$
Hence, $w=0$ in $B_r$.
\\ \\
\emph{Proof of (4).} Let $\epsilon >0$ and $x \in \{x_n=0\}$.  If $\Phi_{x_0}(0^+)>-\infty$ then there exists $\rho$ such that $\Phi_{x_0}(\rho) + \frac{\epsilon}{2} \leq \Phi_{x_0}(0^+)+\epsilon$.  Otherwise, if $\Phi_{x_0}(0^+)=-\infty$ then $\exists$ $m<\infty$ such that $\Phi_{x_0}(\rho) + \frac{\epsilon}{2} \leq -m$.  Now, considering $\rho$ to be fixed, if $|x-x_0|$ is sufficiently small we have $\Phi_x(\rho) \leq \Phi_{x_0}(\rho) + \frac{\epsilon}{2}$.  Also, from Theorem \ref{mono} we have $\Phi_x(0^+) \leq \Phi_x(\rho)$.  Hence we have, 
$$\Phi_x(0^+)\leq
\begin{cases} \Phi_{x_0}(0^+)+\epsilon, & \Phi_{x_0}(0^+)>- \infty \\ -m, & \Phi_{x_0}(0^+)=-\infty
\end{cases}
$$
by first choosing $\rho$ then $|x-x_0|$ small enough.
\end{proof}


\newsec{Angle Conditions and Nonconvexity}{Noncon}

In this section, working in $\R^2$, we will explicitly construct the blowup solutions at free boundary points on the interface to (\ref{pde}).  This classification will lead to an interesting condition on how the free boundary can cross the interface.  In particular it will provide an angle condition that depends only on the given $\alpha$ and $\beta$ from the operator $L_{\alpha, \beta}$.  In other words, the free boundary must cross the interface in a very particular way.  Indeed, we expected the angle of incidence and angle of reflection to be related, but that one angle would be free. However, we find that both angles are completely determined by the given $\alpha$ and $\beta.$ Using this angle condition and the measure convergence in Theorem \ref{conv} we will be able to construct examples of mean value sets that must be nonconvex.

Note that if $w$ is a solution to (\ref{obprob2}) then locally, away from the singularity, $G_k(x,y)-w$ is a solution to (\ref{pde}).  Hence, proposition \ref{deg2} implies that blowup solutions to $G_k(x,y)-w$ at free boundary points are homogeneous of degree two.  If we then restrict to working in $\R^2$ these blowup solutions have the form $r^2g(\theta)$.  Using this fact the next lemma computes such blowup solutions explicitly, in doing so we derive an angle condition for mean value sets as they cross the interface, $\{x_n=0 \}$.

\begin{lemma}[Angle Condition Across the Interface] \label{anglecon}
Let $v_0=r^2g(\theta)$ be a blow-up solution to (\ref{pde}) at the origin and assume $\exists$ $\theta_1$, $\theta_2 \in (0,\pi)$ such that $g(\theta)=0$ $\forall$ $\theta \in [0,\pi-\theta_1] \cup [\pi+\theta_2, 2\pi]$.
 
$$g(\theta):= 
\begin{cases}
\frac{1}{8\alpha}[1-\cos(2\theta_1)\cos(2\theta)+\sin(2\theta_1)\sin(2\theta)] & \quad \pi-\theta_1 \leq \theta \leq \pi \\
\frac{1}{8\beta}[1-\cos(2\theta_2)\cos(2\theta)-\sin(2\theta_2)\sin(2\theta)] & \quad \pi \leq \theta \leq \pi + \theta_2 \\
0 & \quad \text{otherwise}
\end{cases}$$

$$\theta_2=
\begin{cases}
\theta_1+\frac{\pi}{2} & \quad 0<\theta_1<\frac{\pi}{2} \\
\theta_1-\frac{\pi}{2} & \quad \frac{\pi}{2}<\theta_1<\pi \\
\end{cases} \qquad \text{and} \qquad \cos(2\theta_1)=\frac{\beta-\alpha}{\beta+\alpha}$$

\end{lemma}

\begin{proof}
Such a $g$ must satisfy the following conditions:

\begin{enumerate}
\item $g''(\theta)+4g(\theta)= \frac{1}{2 \alpha}$ for $\theta \in (\pi-\theta_1,\pi)$
\item $g''(\theta)+4g(\theta)= \frac{1}{2 \beta}$ for $\theta \in (\pi,\pi + \theta_2)$
\item $g(\pi-\theta_1)=g(\pi+\theta_2)=0$
\item $g'(\pi-\theta_1)=g'(\pi+\theta_2)=0$
\item $\alpha g'(\pi^-)=\beta g'(\pi^+)$
\item $g(\pi^-)=g(\pi^+)$
\end{enumerate}
Conditions 1 and 2 give us
$$ g(\theta):=
\begin{cases}
\frac{1}{8 \alpha}+C_1 \cos(2\theta)+C_2 \sin(2\theta) &\pi- \theta_1<\theta<\pi \\
\frac{1}{8 \beta}+D_1 \cos(2\theta)+D_2 \sin(2\theta) & \pi<\theta<\pi + \theta_2
\end{cases}
$$
Then conditions 3 and 4 give
$$C_1=\frac{-\cos(2\theta_1)}{8\alpha} \quad C_2=\frac{\sin(2\theta_1)}{8 \alpha} \quad D_1=\frac{-\cos(2\theta_2)}{8 \beta} \quad D_2=\frac{-\sin(2 \theta_2)}{8 \beta} $$
Now condition 5 implies
$$-\sin(2\theta_2)=\sin(2\theta_1)$$
which then gives
$$\theta_2=
\begin{cases}
\theta_1+\frac{\pi}{2} & 0 < \theta_1 < \frac{\pi}{2} \\
\theta_1-\frac{\pi}{2} & \frac{\pi}{2} < \theta_1 < \pi
\end{cases}. 
$$
Finally using condition 6 we get
$$\frac{1-\cos(2\theta_2)}{\beta}=\frac{1-\cos(2\theta_1)}{\alpha}$$
Hence we have
$$\cos(2\theta_1)=\frac{\beta-\alpha}{\beta+\alpha}$$.
\end{proof}

\begin{remark} \label{anglecon2}
If in the previous lemma we instead assumed that $\Lambda(v_0)$ was on the ``left" side we get very much the same result.  As in
$$g(\theta):= 
\begin{cases}
\frac{1}{8\alpha}[1-\cos(2\theta_1)\cos(2\theta)+\sin(2\theta_1)\sin(2\theta)] & \quad 0 \leq \theta \leq \pi-\theta_1 \\
\frac{1}{8\beta}[1-\cos(2\theta_2)\cos(2\theta)-\sin(2\theta_2)\sin(2\theta)] & \quad \pi+\theta_2 \leq \theta \leq 2\pi\\
0 & \quad \text{otherwise}
\end{cases}$$
with the same definitions for $\theta_1$ and $\theta_2$.  Hence, assuming $\Lambda(v_0)$ crosses the interface, we only have nine distinct cases.  Six possible blowup solutions come from the angle conditions in Lemma \ref{anglecon} and three are the standard half plane and whole plane solutions with the free boundary on the interface.
\end{remark}

\begin{figure}[!h]
\centering
\begin{subfigure}{.3\linewidth}
    \centering
    \includegraphics[scale=.20]{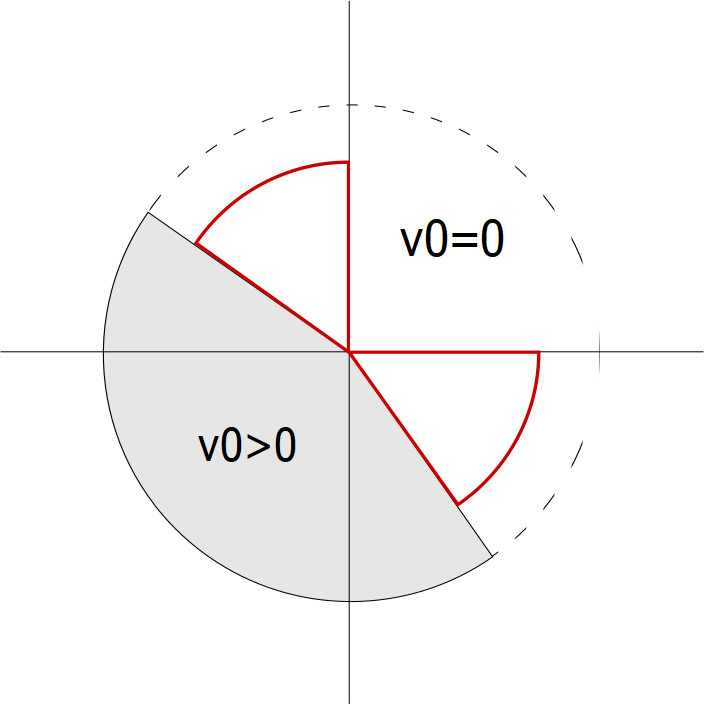}
    \caption{}\label{fig1:image1}
\end{subfigure}
    \hfill
\begin{subfigure}{.3\linewidth}
    \centering
    \includegraphics[scale=.20]{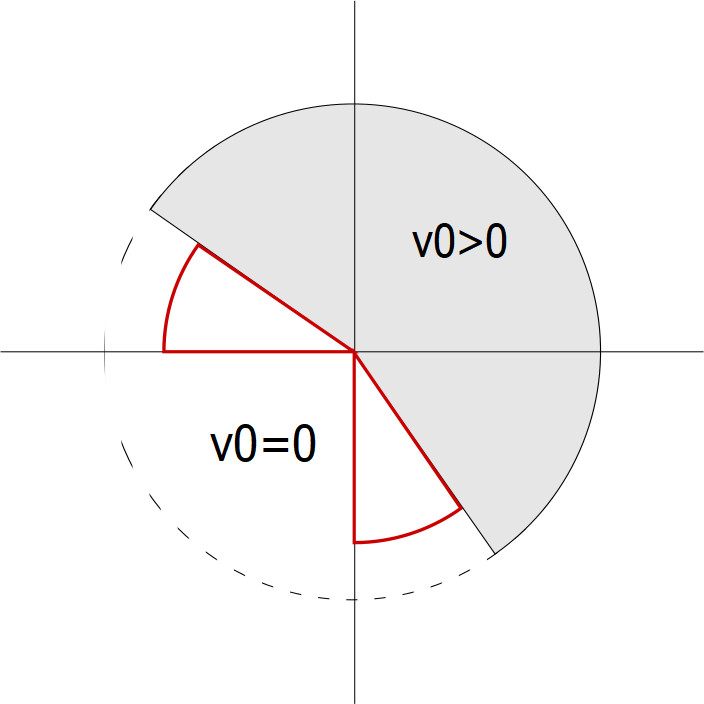}
    \caption{}\label{fig1:image2}
\end{subfigure}
   \hfill
\begin{subfigure}{.3\linewidth}
    \centering
    \includegraphics[scale=.20]{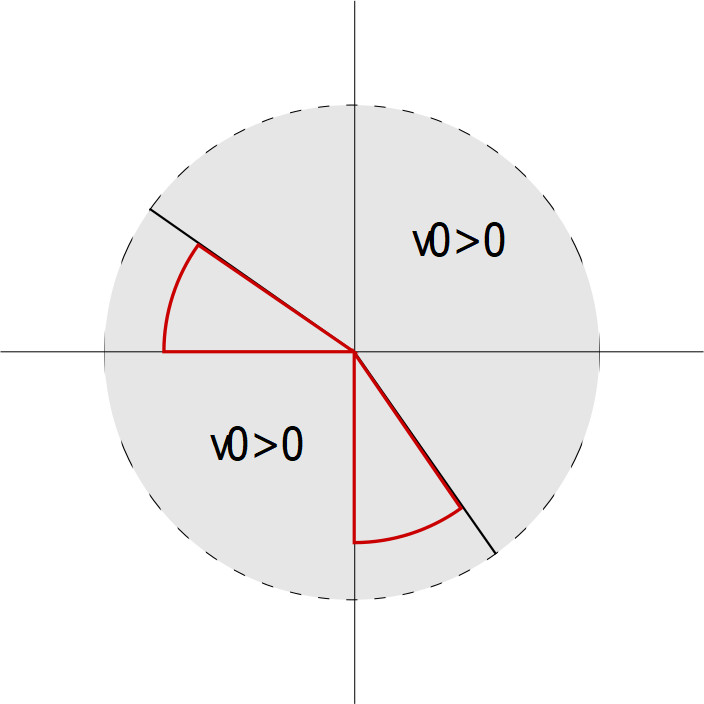}
    \caption{}\label{fig1:image3}
\end{subfigure}
	\vfill
\begin{subfigure}{.3\linewidth}
    \centering
    \includegraphics[scale=.20]{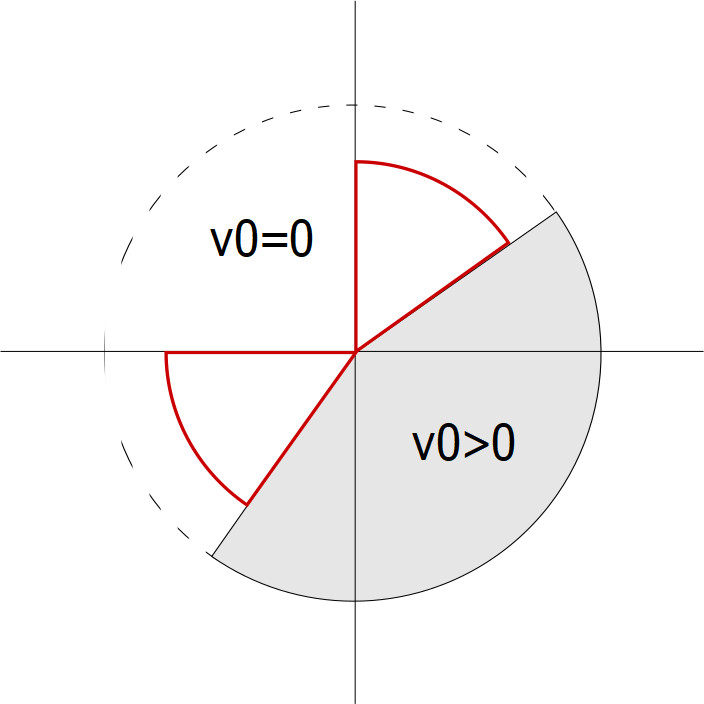}
    \caption{}\label{fig1:image4}
\end{subfigure}
	\hfill
\begin{subfigure}{.3\linewidth}
    \centering
    \includegraphics[scale=.20]{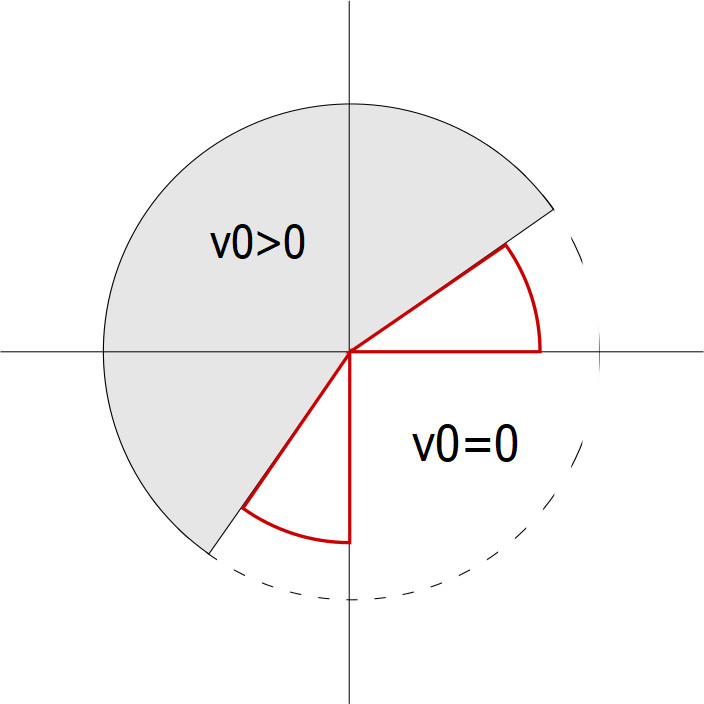}
    \caption{}\label{fig1:image5}
\end{subfigure}
	\hfill
\begin{subfigure}{.3\linewidth}
    \centering
    \includegraphics[scale=.20]{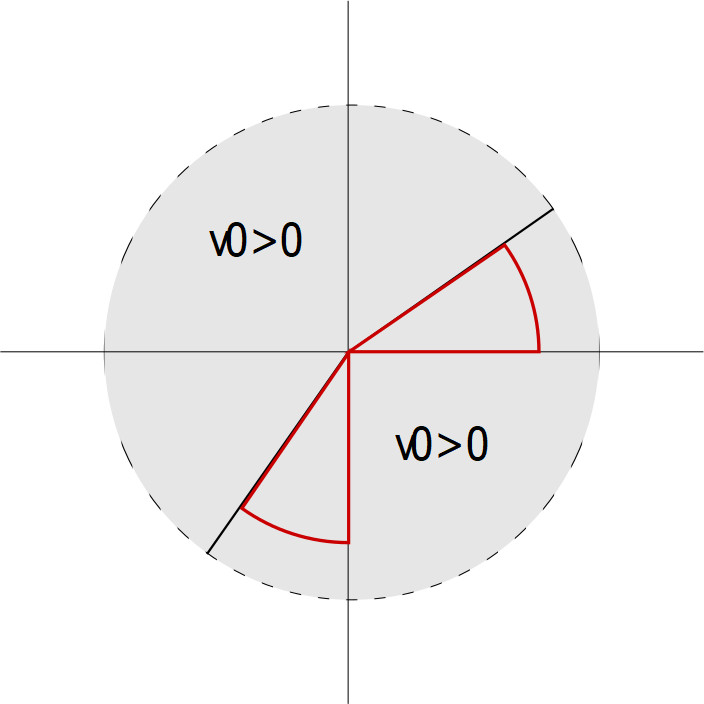}
    \caption{}\label{fig1:image6}
\end{subfigure}
	\vfill
\begin{subfigure}{.3\linewidth}
    \centering
    \includegraphics[scale=.20]{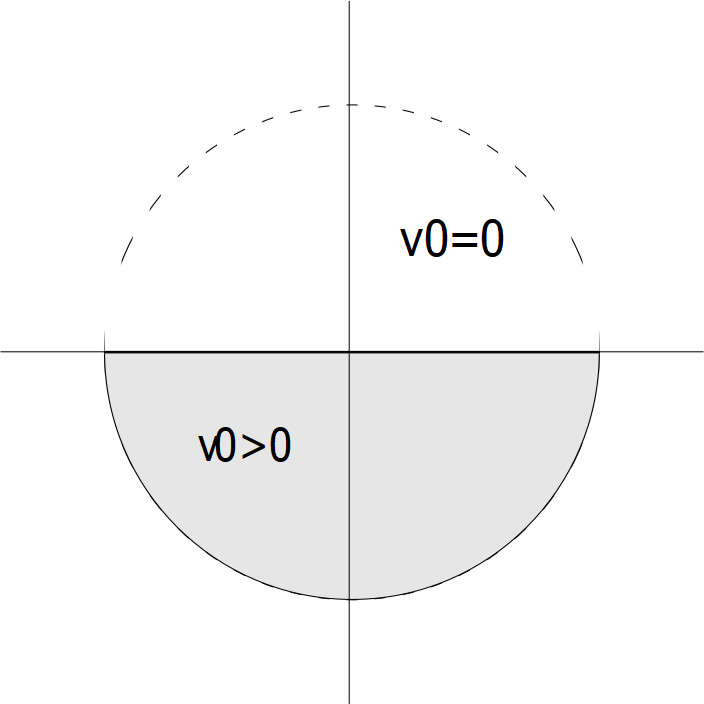}
    \caption{}\label{fig1:image7}
\end{subfigure}
	\hfill
\begin{subfigure}{.3\linewidth}
    \centering
    \includegraphics[scale=.20]{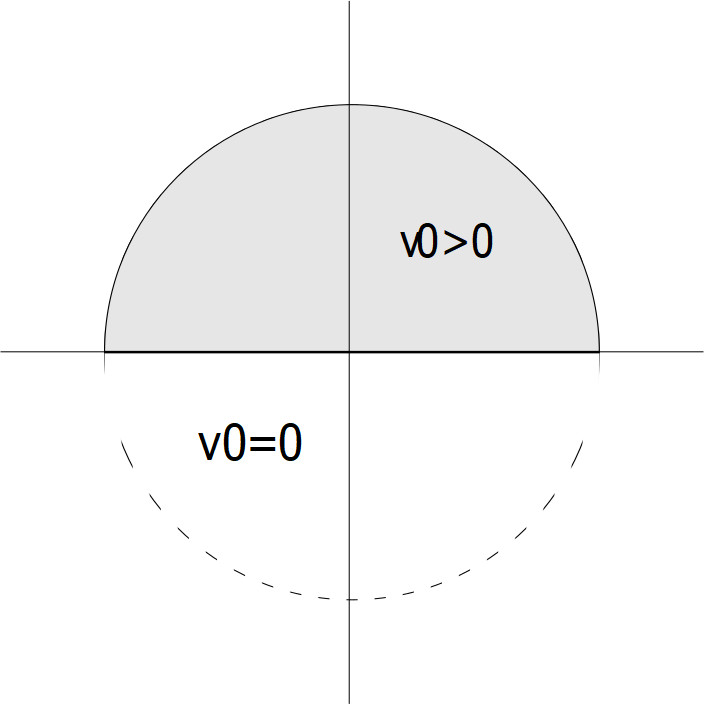}
    \caption{}\label{fig1:image8}
\end{subfigure}
	\hfill
\begin{subfigure}{.3\linewidth}
    \centering
    \includegraphics[scale=.20]{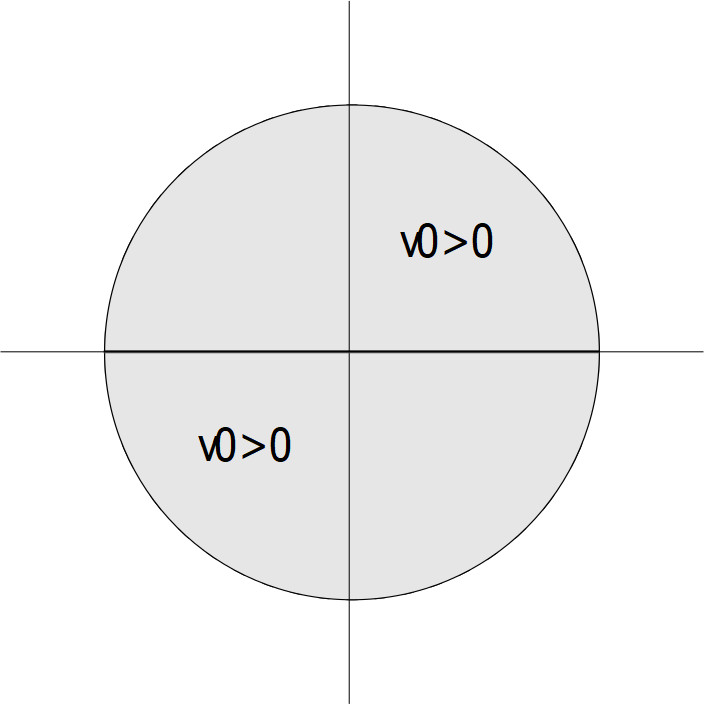}
    \caption{}\label{fig1:image9}
\end{subfigure}	
\caption{The nine possible blowup limits at a free boundary point for fixed $\alpha$ and $\beta$, assuming the positivity set crosses the interface} \label{blowups}	
\end{figure}

At this point it is very easy to believe that convexity of our mean value sets will not always be possible.  The angle condition in Lemma \ref{anglecon} depends only on $\alpha$ and $\beta$ and as $\alpha$ converges to $\beta$ we have $(\theta_1, \theta_2)$ converging to either $(\frac{\pi}{4},\frac{3\pi}{4})$ or $(\frac{3\pi}{4},\frac{\pi}{4})$.  Thus, the boundaries of our mean values sets must satisfy such an angle condition as they cross the interface or become tangent to the interface, but we also have from Theorem \ref{conv} that these mean value sets converge in measure to Euclidean balls as $\alpha$ and $\beta$ converge to one.  Then by picking an appropriate center for the Euclidean balls they will cross the interface with any angle we desire.  This angle condition with the measure convergence statement in Theorem \ref{conv} will force $\partial D_R$ to curve rapidly back towards $\partial B_R$ forcing $D_R$ to be nonconvex.  Before we can prove this we first need a small lemma to ensure the convergence of the contact and noncontact set as we rescale our solutions.

\begin{lemma} \label{FBconv}
Let $v$ be a solution to equation \ref{pde} and $0<\rho_m \rightarrow 0$ be a sequence such that the blowup sequence $v_m(x):=\frac{v(\rho_m x)}{\rho_m^2}$ converges $a.e.$ in $\R^n$ to a blowup limit $v_0$.  Then for every $\epsilon>0$ there exists an $M$ such that for any point $p \in \Omega(v_0)\cap B_1$ that is at least $\epsilon$ away from $FB(v_0)$ we have $p \in \Omega(v_m) \cap B_1$ for all $m \geq M$. Furthermore, for any point $q \in \Lambda(v_0)\cap B_1$ that is at least $\epsilon$ away from $FB(v_0)$ we have $q \in \Lambda(v_m) \cap B_1$ for all $m \geq M$.
\end{lemma}
\begin{proof}
Note that there exists a subsequence of $v_m$, which we again call $v_m$, that converges to $v_0$ uniformly on compact sets.  Also, we have $v_m$ converging to $v_0$ in $C_{loc}^{1,\alpha}(\R_+^n)$ and in $C_{loc}^{1,\alpha}(\R_-^n)$. 

Denote the $\epsilon$-neighborhood of a set $S$ as $N_{\epsilon}(S)$.  We first prove the statement for points in $\Omega(v_0)$.  Outside of $N_{\epsilon}(FB(v_0) \cap B_1)$ in $\Omega(v_0)$ we have $v_0>M$ for some constant $M>0$.  Since $v_m$ converges uniformly to $v_0$ in $B_1$ there exists a $K$ such that $|v_m-v_0|<\delta<M$ in $B_1$.  Hence, $FB(v_m) \cap B_1$ can not be outside of the $\epsilon$-neighborhood of $FB(v_0) \cap B_1$ in $\Omega(v_0)$.

Now we prove the statement for points in $\Lambda(v_0)$.  If we suppose the statement is false, then there exists an $\epsilon>0$ so that for any $m$ there is a point $p_m \in FB(v_m) \cap \Lambda(v_0) \cap B_1$ such that it is at least $\epsilon$ away from $FB(v_0)$.  However, nondegeneracy implies
$$C \bigg( \frac{\epsilon}{2} \bigg)^2 \leq \sup_{x \in B_{\epsilon/2}(p_m)}v_m.$$
Hence, for any $m$ there exists a point $q_m \in \Lambda(v_0)$ such that $v_m(q_m) \geq C( \frac{\epsilon}{2})^2$ which contradicts uniform convergence.
\end{proof}

\begin{proof}[Proof of Theorem \ref{nonconv}]
Without loss of generality we will assume $x_0 = 0.$  Also, for simplicity
we will assume $y_0 = 0.$  This assumption is only to ensure $B_{R}(0,y_0)$
does not satisfy the angle condition in Lemma \ref{anglecon} and to ensure the sets cross the interface. It should be clear from the proof how to adapt this to any $y_0 \neq 0$ as long as $\partial B_R$ does not satisfy the angle condition as it crosses the interface.

We wish to show that there exists a point $q_k := (\tilde{q}_k, 0) \in \partial D_{R;k}(0,0)$ such that
$|-R - \tilde{q}_k| \rightarrow 0$ as $k \rightarrow \infty.$  Note that we know such a point
$q_k \in \partial D_{R;k}(0,0)$ exists due to Theorem \ref{conv} and the fact that $D_{R;k}(0,0)$
has exactly one component.  (See [AB] Lemma 2.5.)  If we assume that $|-R - \tilde{q}_k| \nrightarrow 0$
then since all of the $D_{R;k}(0,0)$ are contained in one fixed large ball we can extract a convergent
subsequence of the $q_k.$  Calling the new subsequence $\{q_k\}$ still, we must have that the density
of $D_{R;k}(0,0)$ at $q_k$ converges to either $1$ or $0$ as $k \rightarrow \infty$ according to
whether the $q_k$ converge to a point inside or outside of $B_R$ respectively.  In either case, it is
clear that the resulting set cannot remain convex while approaching $B_R$ in measure.
See Figure \ref{cusps} where a dotted line is drawn in the zoomed picture which must start and end within
$D_{R;k}(0,0),$ but which (after a slight adjustment if needed) will necessarily contain points of the
complement in order to not violate the measure stability theorem.  Thus, we may assume that
$|-R - \tilde{q}_k| \rightarrow 0.$

\begin{figure}[!h] 
	\centering
	\scalebox{.80}{\includegraphics{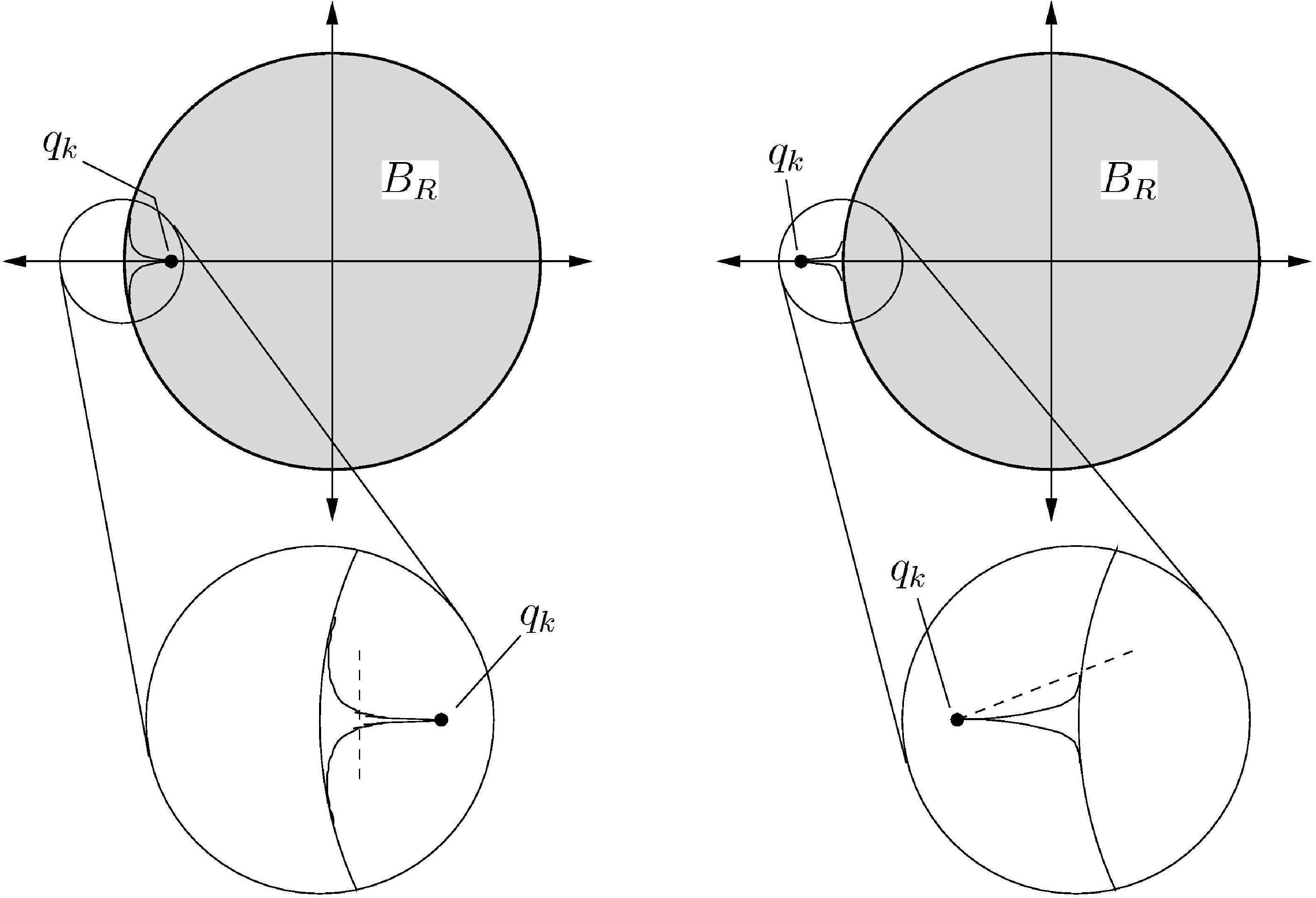}}
	\caption{Cusps when $q_k$ does not converge to $-R$}
	\label{cusps}
\end{figure}

By the argument above, for any arbitrarily small $ \gamma > 0,$ there exists a $K > 0$ such that
$$k \geq K \Rightarrow |-R - \tilde{q}_k| < \gamma.$$  Zooming in at $q_k$ we can be sure that our solution
converges to one of the blowup limits in Figure \ref{blowups}.  Note that blowup solutions as in Figure \ref{fig1:image3}, \ref{fig1:image6}, and \ref{fig1:image9} immediately lead to the nonconvexity of $D_{R;k}(0,0).$  The remainder of these cases can be dealt with similarly, so we will
assume that there is a subsequence converging to a blowup limit of the variety in Figure \ref{fig1:image4} and leave the other cases as an exercise for the interested reader.  Now within this setting, we observe
that $k$ can be chosen large enough so that the following quantities are as small as we like:
\begin{enumerate}
     \item $|-R - \tilde{q}_k|,$
     \item $|\theta_1 - \frac{\pi}{4}|,$ and
     \item $|D_{R;k}(0,0) \, \Delta \, B_R|.$
\end{enumerate}
After fixing $k$ sufficiently large to make the three quantities above adequately small, we can then use
Lemma \ref{FBconv} so that the distance from the free boundary within $B_{\rho}(q_k)$ to the line
with slope one through $q_k$ is shrinking faster than $\rho.$  In short, the zoomed picture within Figure \ref{nonconfig}
is as accurate as we like.  Now however, convexity together with measure stability would give us a
contradiction as we can use convexity to show that in the picture of $B_R$ in the figure, as much of the region
above the dotted line as we want cannot be part of the noncontact set.  On the other hand, the measure of that
region is $\left(\frac{\pi-2}{4}\right)R^2 > 0,$ and that leads to a contradiction with measure stability and our
assumption that $k$ is sufficiently large.

\begin{figure}[!h]
	\centering
	\scalebox{.6}{\includegraphics{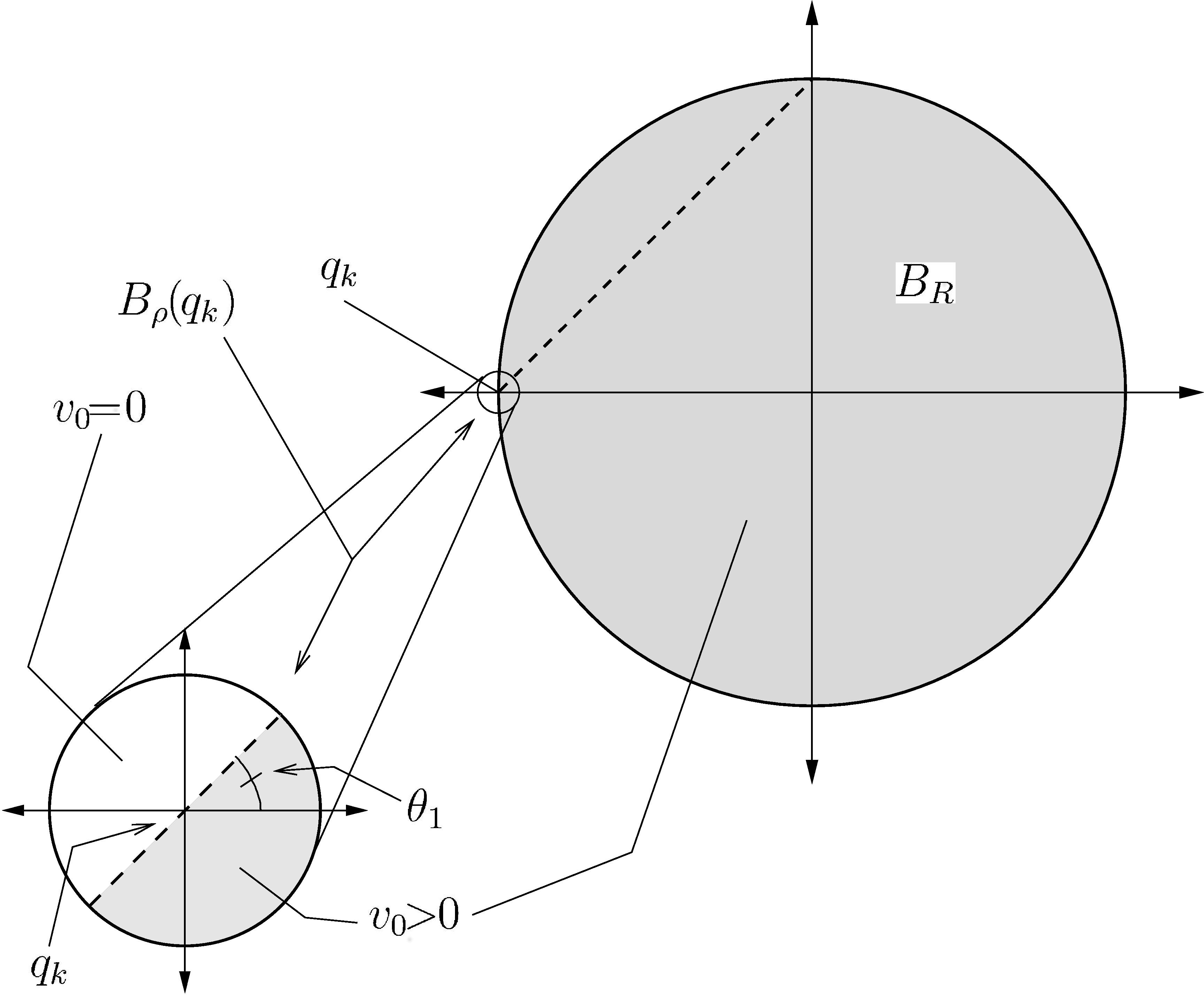}}
	\caption{Set above dotted line contradicts measure stability}
	\label{nonconfig}
\end{figure}
\end{proof}


\newsec{Nonconvexity Continued}{Noncon cont}

In the previous section the discontinuous structure of our $a^{ij}_k(x)$'s seemed to played a major role in the nonconvexity of the associated mean value sets.  Here we will show that in taking smooth approximations of such $a^{ij}_k(x)$'s we will still have nonconvex mean value sets, even though the principal coefficients of the operators are no longer discontinuous.  In light of these examples we now expect that operators whose mean value sets are all convex to be the exception and not the rule.

Let $b^{ij}_{s;k}(x,y)$ be a smooth approximation of $a^{ij}_k(x,y)$ by convolving with a mollifier.  Denote the Green's function associated to $b^{ij}_{s;k}(x,y)$ by $G_{s;k}(x,y)$.  Then from Theorem 5.4 in \cite{LSW} we know that $G_{s;k}(x,y)$ converges uniformly to $G_k(x,y)$ on compact sets away form the singularity as $s \rightarrow \infty.$  Hence, the methods and results in Section 3 can be directly carried over to this setting. In particular we get an analog of Theorem \ref{conv}.

Before proving our final nonconvexity statement we will need a lemma to ensure the regions causing $D_{R;k}(x_0,y_0)$ to be nonconvex are not of measure zero.  Obviously if this were the case then the measure convergence of the mean value sets would not guarantee nonconvexity.  For simplicity we will drop the dependence of $R$ for the notation and simply write $D_k(x_0,y_0)$ and $D_{s;k}(x_0,y_0)$ for the mean value sets associated the operators whose principal coefficients are $a^{ij}_k(x,y)$ and $b^{ij}_{s;k}(x,y)$ respectively.

\begin{lemma} \label{thickcontact}
Let $a^{ij}_k(x,y)$ and $D_k(x_0,y_0)$ be as above.  Then for almost every choice of $y_0,$ there exists a constant $K>0$ such that for all $k>K$ we have an open set $E \subset D_k^c(x_0,y_0)$ with $|E|>0$ and for every point $p \in E$ there exists a line segment containing $p$ starting and ending in $int(D_k(x_0,y_0)).$
\end{lemma}
\begin{proof}
As before we will assume $y_0=0$ for simplicity and without loss of generality assume $x_0=0.$  Again it should be clear that the proof will still hold if $y \neq 0.$  Let $q_k:=(\tilde{q}_k,0) \in D_k(0,0)$ such that $\tilde{q}_k \leq p_k$ for all $(p_k,0) \in \partial D_k(0,0)$.  Now take a convergent subsequence of the $q_k$ calling the new sequence $q_k$ again.  We then have three possible cases, $q_k$ converges either to a point in $int(B^c_R)$, $B_R$, or $\partial B_R.$  

\emph{Case I: $q_k$ converges to a point in $int(B^c_R)$.} \\  
However, $q_k$ can not converge to a point in $int(B^c_R)$ due to the lower bound on the density of $D_k(0,0)$ at the point $q_k$, as shown by Lemma 1.3 in \cite{AB}.

\emph{Case II: $q_k$ converges to a point in $B_R$.} \\
Then there exists $\delta>0$ such that $|-R-q_k| > \delta$ for all $k$. Then we have
$$\frac{|D^c_k(0,0)\cap B_{\delta}(q_k)|}{|B_{\delta}(q_k)|}>0 \quad \text{for all k},$$ 
since otherwise we would contradict the definition of $q_k.$  Note that there is no claim that the above positivity be uniform.   Then by choosing $k$ large enough so that $|(D_k(0,0) \, \Delta \, B_R) \cap \{x \leq -R + \delta \}| \leq \frac{1}{100}|\{x \leq -R+\delta \} \cap B_R|$ the choice of $E:=D^c_k(0,0)\cap B_{\delta}(q_k)$ would satisfy the desired properties.  

\emph{Case III: $q_k$ converges to a point in $\partial B_R$.} \\
As in the proof of Theorem \ref{nonconv} we zoom in at $q_k$ and note that our solution converges to one of the blowup limits in Figure \ref{blowups}.  Note that blowups as in Figures \ref{fig1:image3} and \ref{fig1:image6} are not possible here as this would lead to a contradiction to how $q_k$ was defined.  If the blowup is as in Figure \ref{fig1:image9} we can take $E$ to be defined similarly to that in Case II where $\delta$ need only be small enough.  The remaining blowups are all dealt with similarly so we will assume the blowup limit to be that of the form in Figure \ref{fig1:image4}.  Again as in the proof of Theorem \ref{nonconv} we can choose $k$ sufficiently large so that the zoomed in picture with in Figure \ref{nonconfig} is as accurate as we like.  Hence the choice of $E:=\{(x,y) \in D^c_k(0,0) \cap B_{\rho/2}(q_k) \,| \, x>\tilde{q_k} \}$ would satisfy the desired properties.
\end{proof}

\begin{theorem}
Let $D_k(x_0,y_0)$ and $D_{s;k}(x_0,y_0)$ be as above.  If $D_k(x_0,y_0)$ is nonconvex then there exists an $S>0$ so that $D_{s;k}(x_0,y_0)$ is nonconvex for all $s>S.$
\end{theorem}
\begin{proof}
Again assuming $y_0=0$ for simplicity and without loss of generality assume $x_0=0.$  Let $E$ be as in the previous lemma. Then from Lemma \ref{thickcontact} we have for any $p \in E$ there exists points $q_1, \, q_2 \in int(D_k(0,0))$ such that the line containing $q_1$ and $q_2$ also contains $p.$ Then let $\epsilon_0>0$ be the largest value such that $B_{\epsilon_0}(p) \subset E$. Similarly let $\epsilon_1,$ $\epsilon_2>0$ be the largest values such that $B_{\epsilon_1}(q_1),$ $B_{\epsilon_2}(q_2) \subset int(D_k(0,0))$ and define $\epsilon:=min\{ \epsilon_0,\epsilon_1,\epsilon_2 \}$. Then every line connecting points from $B_{\epsilon}(q_1)$ to $B_{\epsilon}(q_2)$ goes through $B_{\epsilon}(p)$ and every point in $B_{\epsilon}(p)$ is contained in a line that crosses  $B_{\epsilon}(q_1)$ and $B_{\epsilon}(q_2).$  

Using the analog of Theorem \ref{conv} for the sets $D_k(0,0)$ and $D_{s;k}(0,0)$ we can pick $s$ large enough to ensure that almost every point in $B_{\epsilon/2}(q_1)$ and $B_{\epsilon/2}(q_2)$ belongs to $D_{s;k}(0,0)$ and to ensure that almost every point in $B_{\epsilon/2}(p)$ belongs to $D^c_{s;k}(0,0).$ Hence, $D_{s;k}(0,0)$ is nonconvex.
\end{proof}
\newsec{Acknowledgments}{Ack}
I would like to thank my advisor Ivan Blank for not only suggesting I look at the mean value sets associated to the operators found in this paper, but also for providing many enlightening conversations and support. I would also like to thank Nathan Albin for sharing some numerical computations related to this problem with me, and Luis Silvestre for providing some useful feedback.

\end{document}